\documentclass[a4paper,12pt]{amsart}
\usepackage{amssymb,amsthm,amsmath,color,url}
\usepackage[margin=3cm]{geometry}

\usepackage{graphicx}
\usepackage{stmaryrd}
\usepackage[T1]{fontenc}
\usepackage[english]{babel}
\usepackage{comment}
\usepackage{hyperref}
\usepackage{ifthen}
\hypersetup{
  colorlinks   = true,
  linkcolor = red!70!black,
     citecolor    = green!50!black
}
\usepackage[noabbrev,capitalise]{cleveref}
\crefname{subsection}{Subsection}{Subsections}
\crefname{claim}{Claim}{Claims}
\crefname{problem}{Problem}{Problems}

\makeatletter
\def\namedlabel#1#2{\begingroup
   \def\@currentlabel{#2}%
   \label{#1}\endgroup
}
\makeatother

\usepackage{thmtools}
\usepackage{thm-restate}
\usepackage{enumitem}
\usepackage{caption,subcaption}

\usepackage{tikz}
\usetikzlibrary{shadows}
\usetikzlibrary{backgrounds}
\usetikzlibrary{arrows}
\usetikzlibrary{patterns}
\usetikzlibrary{calc}
\usetikzlibrary{shapes}
\usetikzlibrary{positioning}
\usetikzlibrary{math}
\usetikzlibrary{external}
\usetikzlibrary{decorations.pathreplacing}
\declaretheorem[name=Theorem, numberwithin=section]{theorem}
\declaretheorem[name=Lemma, sibling=theorem]{lemma}
\declaretheorem[name=Proposition, sibling=theorem]{proposition}

\declaretheorem[name=Corollary, sibling=theorem]{corollary}

\declaretheorem[name=Claim, sibling=theorem]{claim}
\declaretheorem[name=Claim, numbered=no]{claim*}

\declaretheorem[name=Remark, style=remark, sibling=theorem]{remark}

\def\cqedsymbol{\ifmmode$\lrcorner$\else{\unskip\nobreak\hfil
\penalty50\hskip1em\null\nobreak\hfil$\lrcorner$
\parfillskip=0pt\finalhyphendemerits=0\endgraf}\fi}

\newcommand{\Cay}{\mathrm{Cay}}

\newcommand{\Aut}{\mathrm{Aut}}

\newcommand{\pRS}{\preceq}

\newcommand{\Sep}{(Y,S,Z)}
\newcommand{\Sepp}{(Y',S',Z')}

\newcommand*\Sepi[1]{(Y_{#1},S_{#1},Z_{#1})}

 \newcommand{\Vint}{V_{\mathrm{int}}}
 \newcommand{\Vout}{V_{\mathrm{ext}}}
  
 \newcommand{\cN}{\mathcal{N}}

 \newcommand{\cE}{\mathcal{E}}
    
 \newcommand{\cC}{\mathcal{C}}

 \newcommand{\Zd}{\mathbb Z_2}

\newcommand{\Stab}{\mathrm{Stab}}

\newcommand*\torso[1]{\llbracket #1 \rrbracket}

\newcommand*\sg[1]{\{ #1 \}}

\DeclareMathOperator{\diam}{diam}

\let\le\leqslant
\let\ge\geqslant
\let\leq\leqslant
\let\geq\geqslant

\begin{document}

\title[Structure of quasi-transitive planar graphs]{A note on the structure of locally finite planar quasi-transitive graphs.}

\author[U.~Giocanti]{Ugo Giocanti}
\address[U.~Giocanti]{Theoretical Computer Science Department, Faculty of Mathematics and Computer Science, Jagiellonian University, Krak\'ow,
 Poland}
\email{ugo.giocanti@uj.edu.pl}

\thanks{The author is supported by the National Science Center of Poland
under grant 2022/47/B/ST6/02837 within the OPUS 24 program, and partially supported by the French ANR Project TWIN-WIDTH
(ANR-21-CE48-0014-01).}

\date{}

\begin{abstract}
In an early work from 1896, Maschke   established the complete list of all finite planar Cayley graphs. This result initiated a long line of research over the next century, aiming at characterizing in a similar way all planar infinite Cayley graphs. Droms (2006) proved a structure theorem for finitely generated \emph{planar groups}, i.e., finitely generated groups admitting a planar Cayley graph, in terms of Bass-Serre decompositions. 
As a byproduct of his structure theorem, Droms proved that such groups are finitely presented. More recently, Hamann (2018) gave a graph theoretical proof that every planar quasi-transitive graph $G$ admits a generating $\mathrm{Aut}(G)$-invariant set of closed walks with only finitely many orbits, and showed that a consequence is an alternative proof of Droms' result. Based on the work of Hamann, we show in this note that we can also obtain a general structure theorem for $3$-connected locally finite planar quasi-transitive graphs, namely that every such graph admits a canonical tree-decomposition whose 
edge-separations correspond to cycle-separations in the (unique) embedding of $G$, and in which every part is still quasi-transitive and admits a vertex-accumulation free embedding. This result can be seen as a version of Droms' structure theorem for quasi-transitive planar graphs. As a corollary, we obtain an alternative proof of a result of Hamann, Lehner, Miraftab and R\"uhmann (2022) that every locally finite quasi-transitive planar graph admits a canonical tree-decomposition, whose parts are either $1$-ended or finite planar graphs.
\end{abstract}

\maketitle
  
\section{Introduction}
In his seminal work, Maschke \cite{Maschke}
gave the full list of all finite planar Cayley
graphs. A group admitting a planar Cayley
graph is called a \emph{planar group}, and
Maschke showed in the same paper that the
finite planar groups are exactly the countable
groups of isometries of the $2$-dimensional
sphere $\mathbb S^2$. Based on the works of
Wilkie \cite{Wilkie} and MacBeath
\cite{MacBeath}, Zieschang, Volgt and Coldewey
\cite{Zieschang80} established the complete
list of \emph{planar discontinuous groups},
which are exactly the countable groups for
which there exists a planar Cayley graph with
a \emph{vertex-accumulation-free} planar
embedding\footnote{Note that the basic
definition of planar discontinuous from
Zieschang, Volgt and Coldewey
\cite{Zieschang80} differs from the one we
gave, however it is shown in \cite[Theorems
4.13.11, 6.4.7 and Corrolary 4.13.15]{Zieschang80} that both definitions are
equivalent.}.
We also refer to \cite[Section III.5]{MS83}
for a complementary work on such groups. It is
worth mentioning that every one-ended group is
planar discontinuous,  
hence all the planar groups which do not enter
into the scope of the aforementioned
characterisation are the multi-ended ones,
i.e., the groups whose number of ends is $2$
or $\infty$.

In \cite{Droms}, Droms proved that
finitely generated planar groups are finitely presented, and thus, by a result of Dunwoody \cite{Dunwoody1985}, they are also accessible. In order to do this, he proved a decomposition theorem for such groups, inspired by Stallings' ends theorem  \cite{Sta68}. In group theoretic terms, this result states that every such group admits a finite Bass-Serre decomposition in which all operations involved are special kind of HNN-extensions and free amalgamations, where at each step the parts which are amalgamated correspond to finite subgroups, which are faces of the planar Cayley graphs drawn. 
From a graph theoretic perspective, this theorem intuitively states that every locally finite planar Cayley graph $G=\Cay(\Gamma, S)$ admits a tree-decomposition which is invariant under the action of $\Gamma$, whose parts correspond to planar Cayley graphs of finitely generated subgroups of $\Gamma$, and whose edge-separations correspond to separations induced by cycles in some planar drawing of $G$. 
See \Cref{fig: plan1,fig:plan2} for an illustration of such a decomposition. We show in this note that a decomposition with such properties still exists if we simply assume that $G$ is $3$-connected locally finite quasi-transitive.

\begin{theorem}[Theorem \ref{thm: plan-struct-1}]
\label{thm: intro}
Every planar locally finite $3$-connected
quasi-transitive graph $G$ admits a canonical
tree-decomposition whose edge-separations
correspond to cycle-separations in the
(unique) embedding of $G$, and where every
part is a quasi-transitive subgraph of $G$
admitting a vertex-accumulation-free planar
embedding.
\end{theorem}

From a metric perspective, quasi-transitivity is known to be more general than than the property of being a Cayley graph: in \cite{EFW12}, the authors exhibited a construction of a quasi-transitive graph which is not quasi-isometric to any Cayley graph, answering an initial question of Woess \cite{Woe91}. However, MacManus recently proved that this is not true anymore if we restrict to the class of quasi-transitive graphs which are quasi-isometric to some planar graph \cite{Mac24Planar}; namely, every locally finite quasi-transitive graph which is quasi-isometric to a planar graph is quasi-isometric to some planar Cayley graph. In light of this result, it is then not surprising that planar quasi-transitive graphs should satisfy similar properties as planar Cayley graphs. 
Nevertheless, the proof we give here offers the advantage to be based on purely graph-theoretic arguments, building on the papers \cite{HamannCycle, HamannPlanar}.
Combining Theorem \ref{thm: intro} with Tutte's canonical decomposition of $2$-connected graphs, we obtain as a corollary the following, which was already proved in \cite[Theorem 7.6]{HamannStallings22}, using the machinery of tree-amalgamations. 
 
\begin{corollary}[Corollary \ref{cor: plan-general}]
 \label{cor: intro}
 Let $G$ be a locally finite quasi-transitive planar graph. Then $G$ admits a canonical tree-decomposition of bounded finite adhesion, whose parts are quasi-transitive planar graphs with at most one end.
\end{corollary}

\begin{figure}[htb] 
  \centering 
  \includegraphics[scale=0.5]{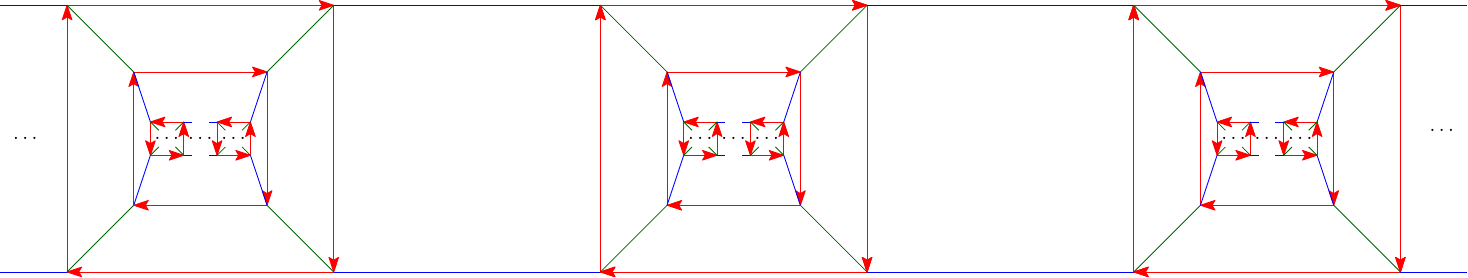}
  \caption{A section of the planar Cayley graph of the group $\langle a,b,c \mid a^4, b^2, c^2, abab, acac\rangle$. The edges associated to $a,b$ and $c$ are respectively colored in red, blue and green. 
  This planar drawing is obtained by drawing a first bi-infinite ladder that corresponds to the Cayley graph of the subgroup generated by $a$ and $b$, and then, in each square face delimited by a cycle $C$ labeled by $a^4$, we draw a copy of the cube, which correponds to the Cayley graph of the subgroup generated by $a$ and $c$, and which we attach along $C$. We repeat this construction an infinite number of steps by alternatively drawing a new ladder or a cube in each new facial cycle delimited by a square labeled by $a^4$.
  The obtained graph then admits a (canonical) tree-decomposition, whose decomposition tree $T$ is the barycentric subdivision of the regular tree with infinite countable degree $\omega$. Each bag corresponding to a node of infinite degree in $T$ contains a bi-infinite ladder, and each bag corresponding to a node of degree $2$ contains a cube. The adhesion sets of this tree-decomposition are cycles of size $4$.}
  \label{fig: plan1}
\end{figure}

\begin{figure}[htb] 
  \centering
  \includegraphics[scale=0.5]{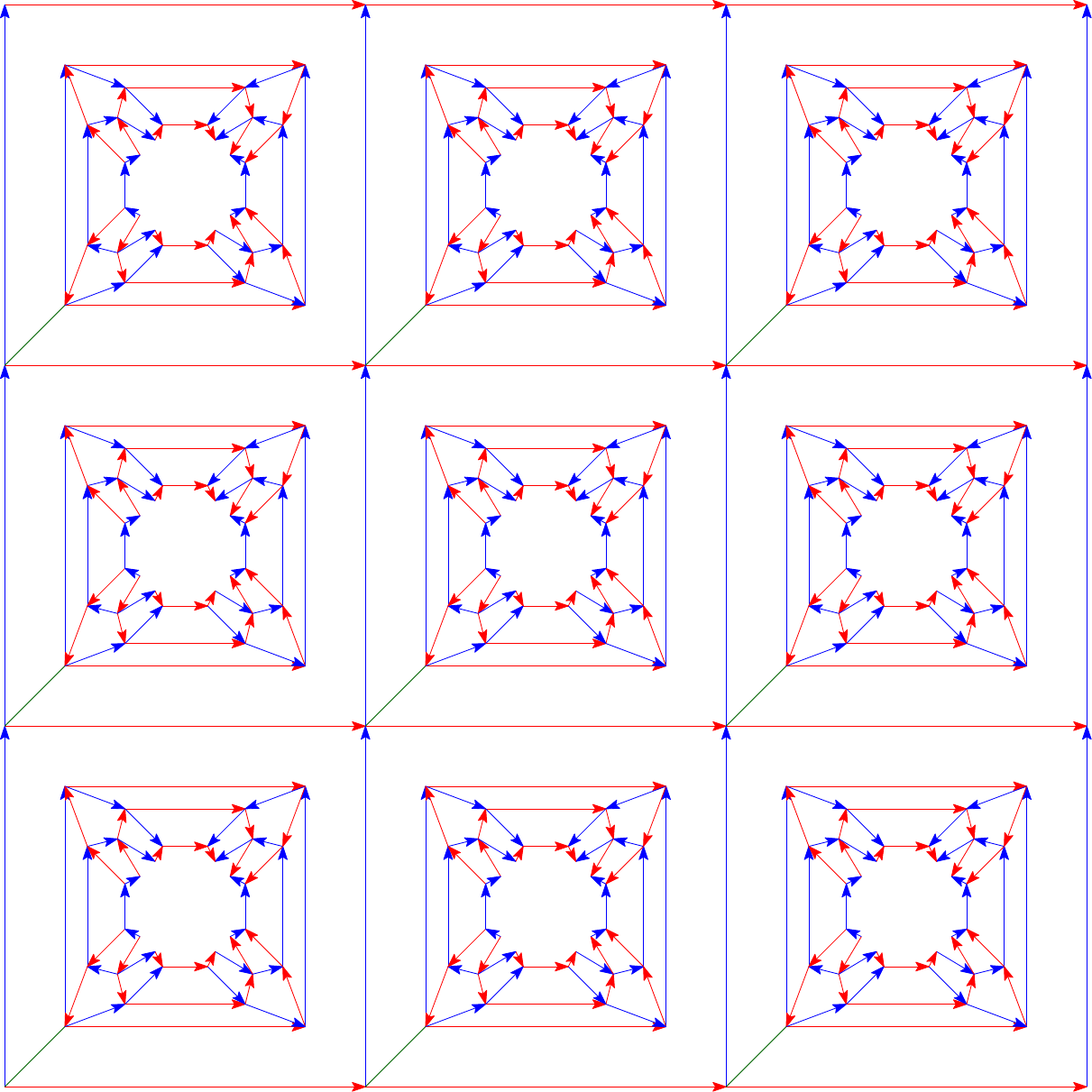}
  \caption{A section of the planar Cayley graph of the group $\langle a,b,c \mid aba^{-1}b^{-1}, c^2\rangle$. This planar drawing is obtained by drawing a first infinite square grid, and then, in each square face, we draw another infinite square grid with a square as outer face, and connect it to the face of the initial grid with an edge. We then keep going on drawing in a similar way a new grid in every square face in the new drawing, and repeat this operation an infinite number of time. A canonical tree-decomposition of the obtained graph satisfying the properties of Corollary \ref{cor: intro} is the one with decomposition tree $T$ the barycentric subdivision of the regular tree with infinite countable degree $\omega$, and where each bag associated to a node of infinite degree contains a copy of the infinite grid, while every bag associated to a node of degree $2$ contains a copy of $K_2$.}
  \label{fig:plan2}
\end{figure}

\paragraph{\textbf{Related work.}} 
A huge amount of work has been done related to the structure of planar quasi-transitive and Cayley graphs. We refer to \cite{Babai97, Renault, Mohar06, Dunwoody07, Georgakopoulos14, GH15, GeoPlanar2conn, GeoCubic, Georgakopoulos_Kleinian, MS22, GH22} for more details on the topic.
    
\section{Preliminaries}
\label{sec: Prel}

\subsection{Quasi-transitive graphs and quasi-isometries}

Let $\Gamma$ be a group acting on a graph $G$ (by automorphisms). 
For every
subset $X\subseteq V(G)$, we let $\Stab_{\Gamma}(X):=\sg{\gamma\in \Gamma:
\gamma\cdot X=X}$ denote the stabilizer of $X$, and for each $x\in X$, we let
$\Gamma_x:=\Stab_{\Gamma}(\sg{x}).$

The action of $\Gamma$ on $G$ is called
\emph{quasi-transitive} if there is only a finite number of  orbits in
$V(G)/\Gamma$. We say that $G$ is 
\emph{quasi-transitive} if it admits a quasi-transitive group action.

For every two graphs $G$ and $H$, a \emph{quasi-isometry} is a map $f: V(G) \rightarrow V(H)$ for which there exists some 
constants $\varepsilon\ge 0$, $\lambda\ge 1$, and $C\ge 0$ such that
(i) for any $y\in V(H)$ there is $x\in V(G)$ such that $d_H(y,f(x))\le C$,
and (ii) for every $x,x'\in V(G)$, $$\frac1{\lambda}d_G(x,x')-\varepsilon\le d_H(f(x),f(x'))\le \lambda d_G(x,x')+\varepsilon.$$
If there exists a quasi-isometry between $G$ and $H$, then we say that $G$ and $H$ are \emph{quasi-isometric} to each other.

\subsection{Separations}

A \emph{separation} in a graph $G=(V,E)$ is a triple $\Sep$
such that $Y,S,Z$ are pairwise-disjoint subsets of $V(G)$, $V=Y\cup S\cup Z$ and there is no edge
between vertices of $Y$ and $Z$. A separation $\Sep$ is \emph{proper}
if $Y$ and $Z$ are nonempty. In this case, $S$ is a \emph{separator} of
$G$. The \emph{order} of a separation $\Sep$ is $|S|$.

The separation $\Sep$ is said to be
\emph{tight} if there are some components $C_Y,C_Z$ respectively of $G[Y],G[Z]$
such that $N_G(C_Y)=N_G(C_Z)=S$. 

The following lemma was originally stated in \cite{TW} for transitive graphs, and the same proof immediately implies that the result also holds for quasi-transitive graphs.
\begin{lemma}[Proposition 4.2 and Corollary 4.3 in \cite{TW}]
\label{lem: TWcut}
 Let $G$ be a locally finite graph. Then for every $v\in V(G)$ and $k\geq 1$, there is 
 only a finite number of tight separations $\Sep$ of order $k$ in $G$ such that $v\in S$.
 Moreover, for any group $\Gamma$  acting
 quasi-transitively on $G$ and  any $k\geq 1$, there is only a
 finite number of $\Gamma$-orbits of tight separations of order at
 most $k$ in $G$. 
\end{lemma}

\subsection{Canonical tree-decompositions}

A \emph{tree-decomposition} of a graph $G$ is a pair $(T,\mathcal V)$ where $T$ is a tree and $\mathcal V=(V_t)_{t\in V(T)}$ is a family of subsets $V_t$ of $V(G)$ such that:
\begin{itemize}
 \item $V(G)=\bigcup_{t\in V(T)}V_t$;
 \item for every nodes $t,t',t''$ such that $t'$ is on the unique path of $T$ from $t$ to $t''$, $V_t\cap V_{t''}\subseteq V_{t'}$; 
 \item every edge $e\in E(G)$ is contained in an induced subgraph $G[V_t]$ for some $t\in V(T)$. 
\end{itemize}

\medskip

The sets $V_t$ for every $t\in V(T)$ are called the \emph{bags} of $(T,\mathcal V)$, and the induced subgraphs $G[V_t]$ the \emph{parts} of $(T,\mathcal V)$. The \emph{width} of $(T,\mathcal V)$ is the supremum of $|V_t|-1$ (possibly infinite), for $t\in V(T)$.
The sets $V_t\cap V_{t'}$ for every $tt'\in E(T)$ are called the \emph{adhesion sets} of $(T,\mathcal V)$ and the \emph{adhesion} of $(T,\mathcal V)$ is the supremum of the sizes of its adhesion sets (possibly infinite). 
The \emph{treewidth} of a graph $G$ is the infimum of the width of $(T,\mathcal V)$, 
among all tree-decompositions $(T,\mathcal V)$ of $G$.

The \emph{torsos} of $(T,\mathcal V)$ are the graphs $G\torso{V_t}$ for $t\in V(T)$, with vertex set $V_t$ and edge set 
$E(G[V_t])$ together with the edges $xy$ such that $x$ and $y$ belong
to a common adhesion set of $(T,\mathcal V)$.

Let $A$ be the set of all the orientations of the
edges of $E(T)$, i.e.\ $A$ contains the pairs $(t_1,t_2)$, $(t_2,t_1)$
for every edge $t_1t_2$ of $T$. For an arbitrary pair $(t_1,t_2)\in
A$, and for each $i\in \{1,2\}$, let $T_i$ denote the component of
$T-\{t_1t_2\}$ containing $t_i$. Then the \emph{edge-separation} of
$G$ associated to $(t_1,t_2)$ is $(Y_1,S,Y_2)$ with $S:= V_{t_1}\cap
V_{t_2}$ and $Y_i:=
\bigcup_{s\in V(T_i)} V_s\setminus S$ for $i\in \sg{1,2}$.

\medskip

For group $\Gamma$ acting on $G$,  we say that a
tree-decomposition $(T,\mathcal V)$ is \emph{canonical with respect to
  $\Gamma$}, or simply \emph{$\Gamma$-canonical}, if $\Gamma$ induces a group action on $T$ such that for every $\gamma
\in \Gamma$ and $t\in V(T)$, $\gamma\cdot V_t= V_{\gamma\cdot t
}$. 
In particular, for every $\gamma\in \Gamma$, note that 
$\gamma$ sends bags of $(T,\mathcal V)$ to bags, and
adhesion sets to adhesion sets. When $(T,\mathcal V)$ is
$\Aut(G)$-canonical, we simply say that it is \emph{canonical}.

\medskip

If $\Gamma$ acts on $G$ and $\mathcal N$ is a family of separations of $G$, we say that
$\mathcal N$ is \emph{$\Gamma$-invariant} if for every $\Sep\in
\mathcal N$ and $
\gamma\in \Gamma$, we have $\gamma \cdot\Sep\in \mathcal N$.
Note that if $(T,\mathcal V)$ is $\Gamma$-canonical, then the associated set of edge-separations is $\Gamma$-invariant.

\smallskip

\begin{remark}
\label{rem: nbaretes}
 If $(T,\mathcal V)$ is a $\Gamma$-canonical tree-decomposition of a locally finite graph $G$ on which $\Gamma$ acts quasi-transitively, whose edge-separations are tight, with finite bounded order, then by Lemma \ref{lem: TWcut} the action of $\Gamma$ on $E(T)$ must induce a finite number of orbits. In particular, $\Gamma$ must also act quasi-transitively on $V(T)$.
\end{remark}

The following two lemmas are folklore results about canonical tree-decompositions. 

\begin{lemma}
 \label{lem: torso-QI}
 Let $G$ be a locally finite $\Gamma$-quasi-transitive graph and $(T, (V_t)_{t\in V(T)})$
 \break be a $\Gamma$-canonical tree-decomposition of $G$ with finite adhesion whose parts are connected subgraphs of $G$ and such that $E(T)$ admits only finitely many $\Gamma$ orbits. Then for every $t\in V(T)$, $G\torso{V_t}$ is quasi-isometric to $G[V_t]$. Moreover, the constants $\varepsilon, \lambda, C$ of the quasi-isometries can be chosen independently of $t\in V(T)$.
\end{lemma}

\begin{proof}
 
 We will show that the identity on $V_t$ induces a quasi-isometry between $G[V_t]$ and $G\torso{V_t}$.
 Let $t\in V(T)$. As $G[V_t]$ is a subgraph of $G\torso{V_t}$, for every $u,v\in V_t$ we have $d_{G\torso{V_t}}(u,v)\leq d_{G[V_t]}(u,v)$. Moreover, 
 as $E(T)$ has finitely many $\Gamma$-orbits and as each part is connected, the set $\sg{d_{G[V_t]}(u,v), \exists s\in N_T(t), u,v\in V_s\cap V_t}$ of values admits a maximum $C_t$. As $V(T)$ has finitely many $\Gamma$-orbits, the set $\sg{C_t: t\in V(T)}$ of values also admits a maximum $C\in \mathbb N$.
 In particular we have $d_{G[V_t]}(u,v)\leq C\cdot d_{G\torso{V_t}}(u,v)$.
\end{proof}

\begin{lemma}
 \label{lem: connected-parts}
  Let $G$ be a connected locally finite $\Gamma$-quasi-transitive graph and 
  \break$(T,(V_t)_{t\in V(T)})$ be a $\Gamma$-canonical tree-decomposition of $G$ with finite adhesion, such that $E(T)$ admits only finitely many $\Gamma$ orbits.
  Then there exists a $\Gamma$-canonical  
  tree-decomposition $(T,(V'_t)_{t\in V(T)})$ of $G$ with finite adhesion, with the same and $\Gamma$-action of $\Gamma$ on $T$, and such that for each $t\in V(T)$, $G[V'_t]$ is connected and quasi-isometric to $G\torso{V_t}$.
\end{lemma}

\begin{proof}
 As $(T, (V_t)_{t\in (T)})$ has finite adhesion and as $E(T)$ has finitely many $\Gamma$-orbits, note that the set $\sg{d_G(u,v): \exists t\in V(T), uv\in E(G\torso{V_t})\setminus E(G)}$ is bounded and thus admits a maximum, say $k\in \mathbb N$. We let $\mathcal V':=(V'_t)_{t\in V(T)}$ be defined by $V'_t:=B_k(V_t)=\sg{v\in V(G): \exists u\in V_t, d_G(u,v)\leq k}$. It is not hard to check that $(T,\mathcal V')$ is also a $\Gamma$-canonical tree-decomposition of $G$ with connected parts. As $G$ has bounded degree, and as $V(T)$ has finitely many $\Gamma$-orbits, $(T,\mathcal V')$ moreover has finite adhesion.
 
 It remains to show that for each $t\in V(T)$, $G\torso{V_t}$ is quasi-isometric to $G[V_t]$.
 For this, we let $t\in V(T)$, and fix any projection $\pi: V'_t\to V_t$ such that $\pi|_{V_t}=\mathrm{id}_{V_t}$ and such that for each $v\in V'_t$, $d_G(\pi(v), v)= d_G(V_t, v)=\min\sg{d_G(v,u): u\in V_t}$.

 We show that $\pi$ defines a quasi-isometry between $G[V'_t]$ and $G\torso{V_t}$. First, note that for every $u,v\in V'_t$ and every $\pi(u)\pi(v)$-path $P$ in $G\torso{V_t}$ of length $d$, there exists a $uv$-path $P'$ in $G[V'_t]$ of length at most $kd+2k$, obtained after replacing each edge of $E(G\torso{V_t})\setminus E(G[V_t])$ in $P$ by a path of size at most $k$, and connecting $u$ to $\pi(u)$ and $v$ to $\pi(v)$ with paths of size at most $k$. 
 Conversely, for every $uv$-path $P'$ of length $d$ in $G[V'_t]$, there exists some $\pi(u)\pi(v)$-path $P$ in $G\torso{V_t}$ of length at most $d\cdot(2k+1)$. This follows from the fact that for every edge $xy\in E(G[V'_t])$, we have $d_{G\torso{V_t}}(\pi(x), \pi(y))\leq 2k+1$.
 
 Hence, for every $u,v\in V_t$ we have 
  $$\frac{1}{k}d_{G[V'_t]}(u, v)-2\leq d_{G\torso{V_t}}(\pi(u),\pi(v))\leq (2k+1)d_{G[V'_t]}(u,v),$$
 thus, as $\pi$ is surjective, it indeed defines a quasi-isometry. 
\end{proof}

\subsection{Rays and ends}

 A \emph{ray} in a graph $G$ is an infinite simple one-way path $P=(v_1,v_2,\ldots)$. A \emph{subray} $P'$ of $P$ is a ray of the form $P'=(v_i,v_{i+1},\ldots)$ for some $i\geq 1$. 
We define an equivalence relation $\sim$ over the set $\mathcal R(G)$ of rays by letting $P\sim P'$ if and only if for every finite set $S\subseteq V(G)$ of vertices, there is a component of $G-S$ that contains infinitely many vertices from both $P$ and $P'$. 
The \emph{ends} of $G$ are the elements of $\mathcal R(G)/\sim$, the equivalence classes of rays under $\sim$. The \emph{degree} of an end $\omega$ is the supremum
number $k\in \mathbb{N}\cup \sg{\infty}$ of pairwise-disjoint rays
that belong to $\omega$. An end is \emph{thin} if it has finite
degree, and \emph{thick} otherwise.

\subsection{Decompositions of quasi-transitive graphs of finite treewidth}
\label{sec: tw}

Without being always explicitly named, bounded treewidth quasi-transitive graphs have attracted a lot of attention and admit many interesting characterisations. Below are a few of them; for more, see for example \cite{KPS, MS83, Woess89,TW, Pichel09}.

\begin{theorem}
\label{thm: tw-ends}
 Let $G$ be a connected $\Gamma$-quasi-transitive locally finite graph. Then the following are equivalent:
 \begin{itemize}
     \item[$(i)$] $G$ has finite treewidth;
     \item[$(ii)$] there exists a $\Gamma$-canonical tree-decomposition of $G$ with tight edge-separations and finite width;
     \item[$(ii)'$] there exists a $\Gamma$-canonical tree-decomposition $(T,\mathcal V)$ of $G$ with finite width, connected parts and such that $E(T)$ has finitely many $\Gamma$-orbits;
     \item[$(iii)$] the ends of $G$ have finite uniformly bounded degree;
     \item[$(iv)$] all the ends of $G$ are thin;
 \end{itemize}
\end{theorem}

We give a short roadmap on a possible way to find a proof of the above equivalences the way we stated them.
The implications $(ii)'\Rightarrow (i)$ and $(iii)\Rightarrow (iv)$ are immediate, and $(ii)\Rightarrow (ii)'$ is an immediate consequence of Lemma \ref{lem: connected-parts}.
It is not hard to see that 
if a graph has an end of degree $k\geq 1$, then it admits the $k\times k$ grid as a minor, thus if $G$ has ends of arbitrary large degree it has infinite treewidth, so $(i)\Rightarrow (iii)$ holds. $(iv)\Rightarrow (ii)$ follows from \cite[Theorem 7.4]{HamannStallings22}.

\subsection{Nested sets of separations}
\label{sec: nested}

We define an order $\pRS$ on the set of separations of a graph $G$ as follows. For any two separations $\Sep, \Sepp$, we write $\Sep \pRS \Sepp$ if and only if $Y\subseteq Y'$ and $Z'\subseteq Z$. 

Two separations $\Sep, \Sepp$ of a graph $G$ are said to be \emph{nested} if $\Sep$ is comparable either with $\Sepp$ or with $(Z', S', Y')$ with respect to the order $\pRS$.
A set $\cN$ of separations of $G$ is \emph{nested} if all its separations are pairwise nested. We say that $\cN$ is \emph{symmetric} if for every $\Sep\in \cN$, we also have $(Z,S,Y)\in \cN$.
It is not hard to observe that if $(T, \mathcal V)$ is a tree-decomposition and $\cN$ denotes its set of edge-separations, then $\cN$ is symmetric and nested. Moreover, if $(T,\mathcal V)$ is $\Gamma$-canonical, then $\cN$ is also $\Gamma$-invariant with respect to the action of $\Gamma$ on the set of separations of $G$.

Extending a known result \cite[Theorem 4.8]{CDHS} for finite graphs, Elbracht, Kneip and Teegen proved in \cite[Lemma 2.7]{EKT20} that symmetry and nestedness together with a third property are sufficient conditions to obtain a tree-decomposition from a nested set of separations. 
  
We say that a set $\cN$ of separations has \emph{finite intervals} if
for every infinite increasing sequence $\Sepi{1}\prec \Sepi{2} \prec \cdots$ of separations from $\cN$, we have 
 $$\bigcap_{i\geq 1}(S_i\cup Z_i)=\emptyset.$$

\begin{theorem}[Lemma 2.7 in \cite{EKT20}]
 \label{thm: CDHS}
 Let $\mathcal N$ be a symmetric nested set of separations with finite intervals in an arbitrary graph $G$.
 Then there exists a tree-decomposition $(T,\mathcal V)$ of $G$ such that the edge-separations of $(T,\mathcal V)$ are exactly the separations from $\mathcal N$ and the correspondence is one-to-one. Moreover, if $\mathcal N$ is $\Gamma$-invariant with respect to some group $\Gamma$ acting on $G$, then $(T,\mathcal V)$ is $\Gamma$-canonical.
\end{theorem} 

\begin{lemma}
\label{lem: bd-degree-star}
 If $G$ is connected, locally finite and $\mathcal N$ is a nested set of separations in $G$ such that for every $\Sep\in \mathcal N$, $S$ has uniformly bounded diameter with respect to the metric $d_G$, then $\cN$ has finite intervals.
\end{lemma}

Note that we do not require in Lemma \ref{lem: bd-degree-star} that the graphs $G[S]$ induced by the separators $S$ are connected, even though it will be the case later as we will apply this lemma on family of separations whose separators are cycles in planar graphs.

\begin{proof}
We let $\Sepi{1}\prec \Sepi{2}\prec \ldots$ denote an infinite sequence of separations from $\cN$. 

As $G$ is connected locally finite, for every finite set $X$ of vertices, $G-X$ has only finitely many connected components, hence there is only a finite number of indices $i\geq 1$ such that $X = S_i$. In particular, we may assume up to taking an infinite subsequence of separations that for every two indices $i\neq j$, $S_i\neq S_j$. Let $A\in \mathbb N$ be an upper bound of the set $\sg{\diam_G(S): \Sep\in \cN}$. Then for every $i\neq j$ such that $S_i\cap S_j\neq \emptyset$, the separator $S_j$ is included in the ball of radius $2A$ around $S_i$. In particular, as $G$ is locally finite, this ball is finite, hence up to taking an infinite subsequence of separations, we may assume moreover that for every $i\neq j$, $S_i\cap S_j=\emptyset$.

Observe that for every two separations $\Sep \prec \Sepp$ such that $S\cap S'=\emptyset$, we have $S'\subseteq Z$. Note that if some vertex $x$ belongs to $Y_i\cup S_i$ for some $i\geq 1$, then we also have $x\in Y_j$ for all $j> i$, hence it will be enough to prove that for every $x\in V(G)$, there exists some $i\geq 1$ such that $x\in Y_i\cup S_i$, in order to conclude that $\cN$ has finite intervals. We consider $x\in Z_1$. Observe that for each $i\geq 1$ such that $x\in Z_{i+1}$, we have $d_G(x,S_{i+1})<d_G(x,S_i)$. Indeed, by previous observation, we have $S_{i+1}\cup Z_{i+1}\subseteq Z_i$, hence $S_{i+1}$ separates $Z_{i+1}$ from $S_i$, and thus every shortest path from $x$ to $S_i$ must intersect $S_{i+1}$. 
In particular, if we set $D:=d_G(x,S_1)$, we must have $x\notin Z_{D+1}$, hence $x\in Y_{D+1}\cup S_{D+1}$, as desired.
\end{proof}

\subsection{Cycle nestedness in plane graphs}

Recall that if a graph $G$ is planar, and $\varphi: G\to \mathbb R^2$ is a planar embedding, then the pair $(G, \varphi)$ is called a \emph{plane graph}. We say that two cycles $C,C'$ in a plane graph $(G,\varphi)$ are \emph{nested} if $\varphi(C')$ does not intersect both the interior and the exterior regions of $C$. 
When $\varphi$ is fixed, we let $\Vint(C)$ (respectively $\Vout(C)$) denote the set of vertices $v\in V(G)$ such that $\varphi(v)$ belongs to the interior (respectively exterior) of $C$.
Then $(\Vint(C),V(C),\Vout(C))$ is a separation of $G$, and if $C$ and $C'$ are nested in $(G,\varphi)$, then $(\Vint(C),V(C),\Vout(C))$ and $(\Vint(C'),V(C'),\Vout(C'))$ are nested with respect to the definition of nestedness we gave in Section \ref{sec: nested}. However note that the converse is not true in general as the fact that $C$ and $C'$ are nested might depend of the planar embedding of $G$ we choose.

Recall that by Whitney's theorem \cite{Whitney}, every $3$-connected planar graph admits a unique embedding in the $2$-dimensional sphere $\mathbb S^2$, up to composition with a homeomorphism of $\mathbb S^2$. Imrich \cite{Imrich} moreover proved that this result also holds in infinite graphs. In particular if $G$ is planar $3$-connected, then for any cycles $C,C'$ both the unordered pair $\sg{\Vint(C),\Vout(C)}$ and the property for $C$ and $C'$ to be nested do not depend on the choice of the planar embedding $\varphi$ of $G$. In this case, we will then not need to precise the planar embedding of $G$ when talking about nestedness. Note also that if $G$ is $3$-connected, then for any pair of cycles $C,C'$ and any automorphism $\gamma\in \Aut(G)$, $C$ and $C'$ are nested if and only if $\gamma\cdot C$ and $\gamma\cdot C'$ are nested.

We say that a set $F\subseteq E(G)$ of edges is \emph{even} if every vertex from $V(G)$ has even degree in the graph $(V(G), F)$.
If we identify a cycle with its sets of edges, then the cycles of $G$ are exactly the inclusionwise minimal finite nonempty sets of edges that are even.
If $(C_1, \ldots, C_k)$ are cycles in $G$, their \emph{$\Zd$-sum} $\sum_{i=1}^k C_i$ is the finite subset of $E(G)$ obtained by keeping every edge appearing in an odd number of $C_i$'s. We let $\cC(G)$ denote the \emph{cycle space} of $G$, that is the $\Zd$-vector space consisting of $\Zd$-sums of cycles of $G$. We say that a subset $\cE$ of $\cC(G)$ \emph{generates} $\cC(G)$ if every element of $\cC(G)$ can be written as a (finite) $\Zd$-sum of elements from $\cE$.

\begin{remark}
\label{rem: even}
 It is well known and not hard to check that elements from $\cC(G)$ correspond exactly to the finite even subsets of $E(G)$.
\end{remark}

\subsection{VAP-free graphs}

Given a plane graph $(G,\varphi)$, an \emph{accumulation point} is a point $x\in \mathbb R^{2}$ that contains infinitely many vertices of $G$ in all its (topological) neighborhoods.
A planar graph $G$ is \emph{vertex-accumulation-free} or simply \emph{VAP-free} if it admits an embedding in $\mathbb R^2$ with no vertex accumulation point, or equivalently an embedding in $\mathbb S^2$ with at most one accumulation point.

A known result that can be deduced from \cite[Lemma 2.3]{Babai97} is that one-ended locally finite planar graphs are VAP-free.
We will show in Theorem \ref{thm: plan-struct-1} that locally finite VAP-free quasi-transitive graphs form the base class of graphs from which we can inductively build all locally finite quasi-transitive planar graphs.
The following is a folklore result about VAP-free graphs.

\smallskip

\begin{proposition}
 \label{prop: VAP-free}
    If $G$ is a quasi-transitive locally finite connected VAP-free graph with at least two ends, then $G$ has bounded treewidth.
\end{proposition}
\begin{proof}
We let $G$ be a locally finite VAP-free graph and $\varphi: G\to \mathbb R^2$ be a VAP-free planar embedding of $G$. 
   
Assume that $G$ has unbounded treewidth. Then by Theorem \ref{thm: tw-ends}, $G$ has a thick end, thus by a recent strengthening of Halin's grid theorem \cite{HalinGrid, GH24}, $G$ contains a subdivision $H$ of the infinite hexagonal grid $\mathbb H$ as a subgraph of $G$. By Whitney's unique embedding theorem \cite{Whitney, Imrich}, the hexagonal grid admits a unique embedding $\varphi_{\mathbb H}$ in the $2$-dimensional sphere $\mathbb S^2$, up to composition with a homeomorphism of $\mathbb S^2$, and thus essentially one VAP-free embedding in $\mathbb R^2$. In particular, it implies that the faces of $(H, \varphi_H)$ are bounded by subdivisions of the faces of $(\mathbb H, \varphi_{\mathbb H})$, and thus the facial cycles of $(H, \varphi_H)$ are finite.
We let $\omega_0$ denote the end of $H$ in $G$, i.e., the set of rays of $G$ that are equivalent to any ray of $H$.
Let $r$ be a ray in $G$. We will show that $r\in\omega_0$, which immediately implies that $G$ has a unique end, as desired. As $G$ is connected, we may assume that its first vertex belongs to $V(H)$. Thus every vertex of $r$ is either in $V(H)$ or drawn in a face of $(H, \varphi|_H)$. As $\varphi$ is a VAP-free embedding, every facial cycle of $(H, \varphi|_H)$ contains only finitely vertices of $G$ in its interior region. 
In particular, $r$ intersects infinitely many times $V(H)$ so we have $r\in \omega_0$.
\end{proof}

\section{Proof of Theorem \ref{thm: intro} and Corollary \ref{cor: intro}.}
 
\subsection{Generating families of cycles}

For every locally finite graph $G$ and every $i\geq 1$, we let $\cC_i(G)$ denote the subset of $\cC(G)$ of cycles that can be written as $\Zd$-sums of cycles of length at most $i$.

\begin{theorem}[Theorem 3.3 in \cite{HamannCycle}]
 \label{thm: plan-nested}
 Let $G$ be a $3$-connected locally finite planar graph and $\Gamma$ be a group acting quasi-transitively on $G$. Then there exists a nested $\Gamma$-invariant set of cycles generating $\cC(G)$. Moreover, for any $i\geq 0$ there exists a $\Gamma$-invariant nested family $\cE_i$ of cycles of length at most $i$ generating $\cC_i(G)$.
\end{theorem}

In the same paper, the author also proved the following result, which can be seen as a generalization of the result of \cite{Droms} that finitely generated planar groups are finitely presented.

\begin{theorem}[Theorem 7.2 in \cite{HamannPlanar}]
 \label{thm: plan-gen}
 Let $G$ be a quasi-transitive planar graph and $\Gamma$ be a group acting quasi-transitively on $G$. Then there exists a $\Gamma$-invariant set $\cE$ of cycles generating $\cC(G)$ with finitely many $\Gamma$-orbits.
\end{theorem}

Despite the fact that the proof of Theorem \ref{thm: plan-gen} from \cite{HamannPlanar} is based on Theorem \ref{thm: plan-nested}, the family which is constructed in Theorem \ref{thm: plan-gen} is not necessarily nested anymore. However we will observe in Corollary \ref{cor: planar} that combining Theorems \ref{thm: plan-nested} and \ref{thm: plan-gen}, we can find in the $3$-connected case a generating family of cycles which is both nested and has finitely many $\Aut(G)$-orbits.

\smallskip

The following is a basic fact in homology, and comes from a remark of Matthias Hamann (private communication).

\begin{remark}
 \label{rem: Z-sums}
Theorem \ref{thm: plan-gen} was stated in \cite{HamannPlanar} in a more general way for $\mathbb Z$-sums of oriented cycles, i.e., formal sums of oriented cycles with coefficients in $\mathbb Z$.
To see that \cite[Theorem 7.2]{HamannPlanar} implies Theorem \ref{thm: plan-gen} the way we stated it, observe that if a cycle can be written as a formal sum $\alpha_1\overrightarrow{C_1} + \cdots + \alpha_k\overrightarrow{C_k}$ of oriented cycles with coefficients $\alpha_i\in\mathbb Z$, then it can also be written in $\cC(G)$ as the $\Zd$-sum of the cycles $C_i$ such that $\alpha_i$ is odd.
\end{remark}

We observe that in the $3$-connected case, one can find a generating family $\cE$ of cycles combining both the properties of Theorems \ref{thm: plan-nested} and \ref{thm: plan-gen}.

\begin{corollary}
 \label{cor: planar}
 Let $G$ be a locally finite $3$-connected planar graph and $\Gamma$ be a group acting quasi-transitively on $G$. Then there exists a $\Gamma$-invariant set of cycles generating $\cC(G)$ which is nested and has finitely many $\Gamma$-orbits.
\end{corollary}

An example of a family satisfying the properties of Corollary \ref{cor: planar} is given in Figure \ref{fig:plan1} below.

\begin{proof}
 
 We let $\cE$ be a $\Gamma$-invariant family of cycles generating $\cC(G)$ with finitely many $\Gamma$-orbits given by Theorem \ref{thm: plan-gen}. Then in particular there is a bound $K\geq 0$ on the size of the cycles from $\cE$. By Theorem \ref{thm: plan-nested}, there exists a nested $\Gamma$-invariant family $\cE'$ of cycles of lenth at most $K$ in $G$ generating the set $\cC_K(G)$. In particular, $\cE'$ also generates the whole cycle space $\cC(G)$. 
 As $G$ has bounded degree, every vertex $v\in V(G)$ belongs to only finitely many cycles of size at most $K$. In particular, as $\Gamma$ acts quasi-transitively on $V(G)$, it implies that $\Gamma$ also acts quasi-transitively on the set of cycles of size at most $K$ in $G$. Thus $\cE'$ has finitely many $\Gamma$-orbits and satisfies the desired properties. 
\end{proof}

\begin{figure}[htb]
  \centering
  \includegraphics[scale=0.85]{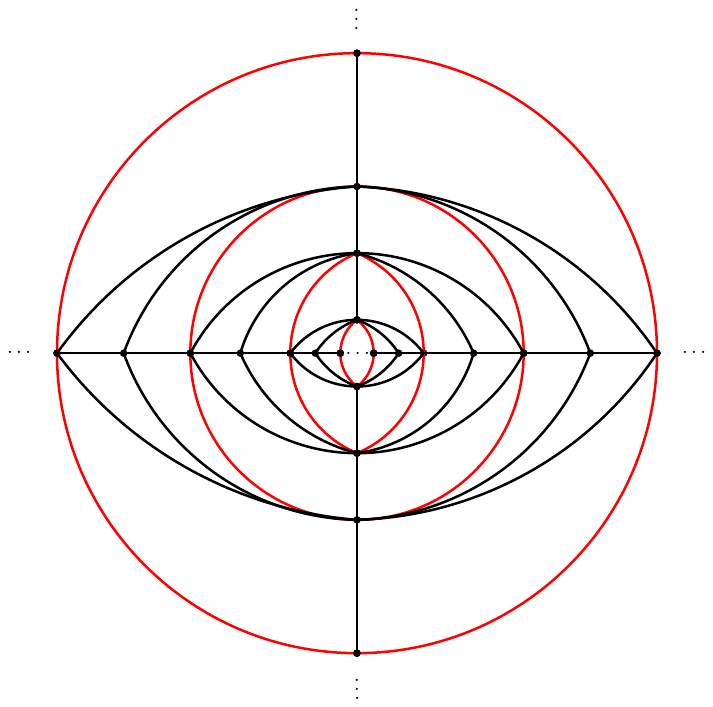}
  \caption{A $2$-ended locally finite quasi-transitive $3$-connected planar graph $G$. The set $\cE$ of cycles formed by the union of the red cycles together with the set of facial cycles of $G$ forms a nested $\Aut(G)$-invariant generating family of the cycle space of $G$. The subgraphs induced by the $\cE$-blocks are the subgraphs obtained after taking two consecutive red cycles, together with the vertices and edges lying between them.}
  \label{fig:plan1}
\end{figure}

\subsection{Structure of locally finite quasi-transitive planar graphs}

If $\cN$ is a set of separations, an \emph{$\cN$-block} is a maximal set $X\subseteq V(G)$ such that for each $\Sep\in \cN$, either $X\cap Y=\emptyset$ or $X\cap Z=\emptyset$.
If $(G,\varphi)$ is a plane graph and $\cE$ is a set of cycles, then an \emph{$\cE$-block} of $(G,\varphi)$ is a set of vertices which is an $\cN$-block, where $\cN$ denotes the symmetric set of separations induced by $\cE$ in $(G,\varphi)$.
 
\begin{lemma}
 \label{lem: part-VAP-free}
 Let $(G,\varphi)$ be a $3$-connected locally finite plane graph, $\Gamma$ be a group acting quasi-transitively on $G$ and $\cE$ be a $\Gamma$-invariant nested family of cycles of bounded length generating the cycle space $\cC(G)$. Then for each $\cE$-block $X$, the family $\cE_X:=\sg{C\in \cE: V(C)\subseteq X}$ generates the cycle space $\cC(G[X])$.
\end{lemma}

\begin{proof}
 In this proof, we will identify every cycle of $G$ with its even set of edges.

 We let $C$ be a cycle of $\cC(G[X])$ and $C_1,\ldots, C_k\in \cE$ be such that $C$ equals to the $\Zd$-sum $\sum_{i=1}^k C_i$. Choose $C_1, \ldots, C_k$ that minimize the number $k$ of cycles from $\cE$ required to write $C$ as a $\Zd$-sum $\sum_{i=1}^kC_i$.
 As $X$ is an $\cE$-block of $(G,\varphi)$, every cycle from $\cE_X$ must be facial in the plane graph $(G[X], \varphi|_{G[X]})$, and $C$ is nested with every cycle from $\cE$. 
 We will show that $C_i\in \cE_X$ for all $i\in[k]$, implying the desired result. 
 
  Assume for a contradiction that $C_i\notin \cE_X$ for some $i\in [k]$. Then as $X$ is an $\cE$-block, there exists some cycle $C^*\in \cE$ separating $X$ from $C_i$, and one of its two associated separations $(Y,S,Z)$ is such that $X\subseteq Y\cup S$ and $V(C_i)\subseteq S\cup Z$. In particular, as $V(C)\subseteq X$, it implies that the cycles $C_i$ and $C$ are not drawn in the same face of $C^*$, hence there are two separations $\Sepi{1}, \Sepi{2}$ respectively associated to $C$ and $C_i$ such that 
  $$\Sepi{1}\pRS \Sep \pRS \Sepi{2}.$$
  
  As cycles from $\cE$ have bounded length, by Lemma \ref{lem: bd-degree-star}, the set $\cN$ of separations induced by $\cE\cup\sg{C}$ in $(G,\varphi)$ has finite intervals. In particular, there are only finitely many cycles $C^*$ separating $X$ from $C_i$ and we can choose such a cycle $C^*$ which is minimal with respect to $\pRS$.

   \begin{claim}
  \label{clm: Ex}
  We have $C^*\in \cE_X$.
 \end{claim}

 \begin{proof}[Proof of Claim]
  Assume for a contradiction that 
  $C^*\notin \cE_X$. Then $V(C^*)\not\subseteq X$, and as $X$ is an $\cE$-block, there exists some cycle $C'\in \cE$ separating $X$ from $C^*$. In particular, $C'$ also separates $X$ from $C_i$ and one of its two associated separations $\Sepp$ satisfies 
  $$\Sepi{1}\pRS \Sepp \pRS \Sep,$$
  contradicting the minimality of 
  $\Sep$.
 \end{proof}

  We now let $D$ be the cycle associated to the maximal separation $\Sepp\in \cE$ such that $\Sep\pRS \Sepp\pRS \Sepi{2}$ and such that $D\in \cE_X$.
 In particular, by previous observation $D$ is facial in $(G[X], \varphi|_{G[X]})$ so $\varphi(C_i)$ must be drawn in the closure of the face $\Lambda$ of $(G[X], \varphi|_{G[X]})$ which is delimited by $D$. We let $I\subseteq [k]$ be the set of indices $j\in [k]$ such that $\varphi(C_j)$ is contained in the closure of $\Lambda$. In particular, $i\in I$ so $I\neq \emptyset$. As $\cE$ is nested, for every $j\in [k]\setminus I$, $\varphi(C_j)$ does not intersect $\Lambda$. Let $C'$ be the $\Zd$-sum $\sum_{j\in I}C_j$. Note that the way we defined it, $C'$ is a finite subset of edges of $E(G)$ but not necessarily a cycle of $G$.

 First, note that for each $uv\in E(G)\setminus E(G[X])$, as $uv\notin C$, it must appear in an even number of $C_j$'s.
 In particular, as we assumed that $I\neq \emptyset$, 
 we must have $|I|\geq 2$.
 In particular, for every $uv\in E(G)$ such that $\varphi(uv)$ intersects $\Lambda$, $uv$ can only appear in cycles $C_j$ such that $j\in I$, and its total number of occurences in $(C_j)_{j\in [k]}$ is even, so $uv\notin C'$. 
 It implies that $C'\subseteq D$. By Remark \ref{rem: even}, $C'$ is even so as $D$ is a cycle, we have either $C'=\emptyset$ or $C'=D$. According to whether $C'=\emptyset$ or $C'=D$, we consider the decomposition of $C$ as a sum of cycles from $\cE$ obtained after either removing the sum $\sum_{j\in I}C_j$ in the decomposition of $C$ or replacing it by the cycle $D\in \cE_X$. In both cases it gives a decomposition of $C$ involving at most $k-|I|+1<k$ cycles from $\cE$, and thus contradicting the minimality of $k$.
\end{proof}

We are now ready to give a proof of Theorem \ref{thm: intro}.

\begin{theorem}
 \label{thm: plan-struct-1}
 Let $G$ be a locally finite $3$-connected quasi-transitive planar graph and $\Gamma$ be a group acting quasi-transitively on $G$. Then there exists a $\Gamma$-canonical tree-decomposition $(T,\mathcal V)$ of $G$ of finite adhesion whose edge-separations correspond to separations associated to cycles of $G$ and whose parts are $2$-connected VAP-free quasi-transitive graphs. Moreover $E(T)$ has finitely many $\Gamma$-orbits.
\end{theorem}

\begin{proof} 
 We let $\cE$ be a nested $\Gamma$-invariant family of cycles of $G$ generating $\cC(G)$ with finitely many $\Gamma$-orbits given by Corollary \ref{cor: planar}. We consider the associated symmetric family $\cN$ of separations of $G$ of the form $(\Vint(C), V(C), \Vout(C))$
 and $(\Vout(C), V(C), \Vint(C))$ for each $C\in \cE$.
 As $G$ is $3$-connected, our previous remarks imply that $\cN$ is a $\Gamma$-invariant nested family of separations. Moreover, as $\cE$ has finitely many $\Gamma$-orbits, separations in $\cN$ must have finite bounded order so Lemma \ref{lem: bd-degree-star} implies that $\cN$ has finite intervals. We thus can apply Theorem \ref{thm: CDHS} and find a $\Gamma$-canonical tree-decomposition $(T,\mathcal V)$ whose edge-separations are in one-to-one correspondence with the different separations of $\cN$. In particular, each adhesion set of $(T,\mathcal V)$ admits a spanning cycle from $\cN$ and thus is finite. As $\cN$ has finitely many $\Gamma$-orbits, $\Gamma$ acts quasi-transitively on $E(T)$. By \cite[Lemma 3.13]{EGL23}, $\Gamma_t$ induces a quasi-transitive action on the part $G[V_t]$ of $(T,\mathcal V)$ for every $t\in V(T)$.
 
 Note that as $G$ is connected, the torsos of $(T,\mathcal V)$ must be connected. 
 Moreover, as the adhesion sets of $(T,\mathcal V)$ are connected, each part $G[V_t]$ must also be connected. Moreover, note that as adhesion sets of $(T,\mathcal V)$ contain spanning cycles, then for every $t\in V(T)$, $|V_t|\geq 3$ and for any three different vertices $u,v,w\in V_t$, any path in $G$ from $u$ to $v$ avoiding $w$ can be modified to a path in $G[V_t]$ from $u$ to $v$ avoiding $w$. Hence each part of $(T,\mathcal V)$ is $2$-connected. 

 It remains to show that each part of $(T,\mathcal V)$ is VAP-free. By \cite[Theorem 4.8]{CDHS}, parts of $(T,\mathcal V)$ are either ``hubs'', i.e., vertex sets of cycles from $\cE$, or $\cN$-blocks\footnote{Note that \cite[Theorem 4.8]{CDHS} only deals with finite graphs. However, the tree-decomposition given by \cite[Lemma 2.7]{EKT20} generalizes the construction from \cite{CDHS} when one considers nested sets of separations having finite intervals, and the proof that its bags are either hubs or blocks extends in this case.} (and equivalently $\cE$-blocks). Hubs parts are finite and thus obviously VAP-free. Assume now that $G[V_t]$ is an $\cE$-block for some $t\in V(T)$. Then by Lemma \ref{lem: part-VAP-free}, $\cE_{V_t}$ generates the cycle space $\cC(G[V_t])$. In particular, note that cycles from $\cE_{V_t}$ must be facial in the plane graph $(G[V_t], \varphi|_{G[V_t]})$. The plane graph $(G[V_t], \varphi|_{G[V_t]})$ is thus $2$-connected and its cycle space is generated by a family of facial walks, so by \cite[Theorem 7.4]{Thomassen80} it must be a VAP-free graph.
\end{proof}

Combining Theorem \ref{thm: plan-struct-1} with results from \cite{CTTD} allowing to combine canonical tree-decompositions, we obtain the following result for $3$-connected planar graphs. 

\begin{corollary}
\label{cor: plan-3conn}
 For every locally finite $3$-connected quasi-transitive planar graph $G$, and every group $\Gamma$ acting quasi-transitively on $G$, there exists a $\Gamma$-canonical tree-decomposition $(T,\mathcal V)$ of $G$ of finite adhesion whose parts are connected and either finite or quasi-transitive one-ended, and such that $E(T)$ has finitely many $\Gamma$-orbits.
\end{corollary}
 
\begin{proof}
 Let $(T,(V_t)_{t\in V(T)})$ be the $\Gamma$-canonical tree-decomposition of $G$ given by Theorem \ref{thm: plan-struct-1}. We let $t\in V(T)$ be such that $V_t$ is infinite. 
 If $G[V_{t}]$ has at least $2$ ends, then Proposition \ref{prop: VAP-free} implies that $G[V_{t}]$ has bounded treewidth. By Lemma \ref{lem: torso-QI}, $G\torso{V_{t}}$ is quasi-isometric to $G[V_{t}]$, so as the property of having finite treewidth in bounded degree graphs is invariant under taking quasi-isometries, $G\torso{V_{t}}$ also has bounded treewidth. By Theorem \ref{thm: tw-ends} $(ii)'$, there exists a $\Gamma_{t}$-canonical tree-decomposition $(T_{t}, \mathcal V_{t})$ of $G\torso{V_{t}}$ of finite width whose parts are connected, and such that $E(T_{t})$ has finitely many $\Gamma_{t}$-orbits. Then by \cite[Proposition 7.2]{CTTD} 
 (see \cite[Proposition 3.10, Remark 3.11]{EGL23} for a statement of this result closer to the one we use here), there exists a $\Gamma$-canonical tree-decomposition $(T', \mathcal V')$ of $G$ \emph{refining} $(T,\mathcal V)$, whose torsos are connected with at most one end, and whose adhesion sets are either adhesion sets of $(T,\mathcal V)$ or adhesion sets of some $(T_t, \mathcal V_t)$.
 In particular, as $G$ is locally finite quasi-transitive, every finite set is the separator of a finite bounded number of separations, hence $E(T')$ must have only finitely many $\Gamma$-orbits. 
 Finally, we find a tree-decomposition of $G$ with the desired properties by applying Lemma \ref{lem: connected-parts} to $(T',\mathcal V')$.
 The fact that its parts are quasi-transitive follows from \cite[Lemma 3.13]{EGL23}.
\end{proof}

See Figure \ref{fig:plan3} below for an illustration of the tree-decomposition obtained (which turns out to be a path-decomposition in this specific example) when applying the proof of Theorem \ref{thm: plan-struct-1} with respect to the family of cycles from Figure \ref{fig:plan1}.

\begin{figure}[htb] 
  \centering
  \includegraphics[scale=0.74]{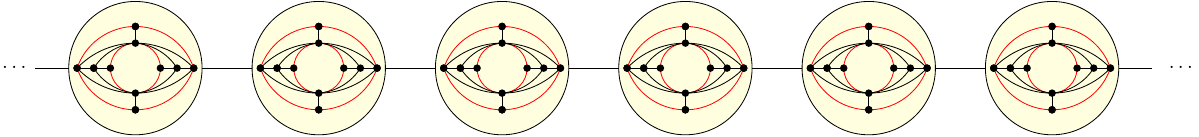}
  \caption{The path-decomposition of the graph from Figure \ref{fig:plan1} obtained after applying Theorem \ref{thm: CDHS} to the nested family of red cycles. The red cycles form the adhesion sets of the path-decomposition.}
  \label{fig:plan3}
\end{figure}

To derive a structure theorem for general planar graphs, we will need to decompose graphs with connectivity at most $2$ into parts of larger connectivity. To do this, we will use a general decomposition theorem, initally proved by Tutte \cite{Tutte} in the finite case, and generalized to infinite graphs in \cite{DSS98}. See also \cite[Theorem 1.6.1]{CK23} for a more precise version.

\begin{theorem}[\cite{DSS98}]
 \label{thm: Tutte}
 Every  locally finite graph $G$ has  a canonical tree-decompo\-sition  of adhesion at most $2$ with tight edge-separations, whose torsos are minors of $G$ and are complete graphs of order at most $2$, 
 cycles, or $3$-connected graphs.
\end{theorem}

Our proof of Corollary \ref{cor: intro} will now simply follow from a combination of Corollary \ref{cor: plan-3conn} with Theorem \ref{thm: Tutte}.

\begin{corollary}
\label{cor: plan-general}
 For every connected planar locally finite graph $G$, and every group $\Gamma$ acting quasi-transitively on $G$, there exists a $\Gamma$-canonical tree-decomposition $(T,\mathcal V)$ of finite adhesion whose parts $G[V_t]$ are connected, either finite or one-ended planar graphs, on which $\Gamma_t$ acts quasi-transitively for each $t\in V(T)$, and such that $E(T)$ has finitely many $\Gamma$-orbits.
\end{corollary}

\begin{proof}

We first consider Tutte's canonical tree-decomposition $(T_0,\mathcal V_0)$ of $G$ given by Theorem \ref{thm: Tutte}. We let $G^+$ be the supergraph obtained from $G$ after adding an edge $uv$ for each pair of vertices $u,v$ belonging to a common adhesion set of $(T_0,\mathcal V_0)$. In particular for each $t\in V(T_0)$, $G\torso{V_t}=G^+[V_t]$.
As the edge-separations of $(T_0,\mathcal V_0)$ are tight, Lemma \ref{lem: TWcut} implies that for each $v\in V(G)$, there is only a finite 
bounded number of edges $tt'\in E(T_0)$ such that $v\in V_t\cap V_{t'}$. In particular, for each $v\in V(G)$, there is only a finite bounded number of $t\in V(T_0)$ such that $v\in V_t$. Thus $G^+$ is also locally finite, and as $(T_0,\mathcal V_0)$ is $\Gamma$-canonical, $\Gamma$ also acts quasi-transitively on $G^+$. We will now show that $G^+$ is planar. Note that $(T_0,\mathcal V_0)$ also corresponds to Tutte's decomposition of $G^+$, and as every torso $G\torso{V_t}$ of $(T_0,\mathcal  V_0)$ is a minor of $G$, every torso $G\torso{V_t}$ is planar. By a result attributed to Erd{\H{o}}s (see for example \cite{Thomassen80}), 
a countable graph is planar if and only it excludes $K_{3,3}$ and $K_5$ as minors. In particular it implies that a countable graph is planar if and only if all its finite subgraphs are planar, so it is enough to check that every finite subgraph of $G^+$ is planar. As the adhesion sets of $(T_0, \mathcal V_0)$ induce complete graphs in $G^+$ and have size at most $2$, note that any finite subgraph of $G^+$ is obtained after perfoming the following operation a finite number of times: taking two disjoint planar graphs $G_1, G_2$ with two edges $u_1v_1\in E(G_1)$ and $u_2v_2\in E(G_2)$, and gluing them by identifying $u_1$ with $u_2$ and $v_1$ with $v_2$. Note that such an operation does preserve planarity, hence every finite subgraph of $G^+$ must be planar
and we deduce that $G^+$ is also planar.  

\begin{claim}
 \label{clm: torsos-plan}
 If $G^+$ admits a $\Gamma$-canonical tree-decomposition $(T,\mathcal V)$ of finite adhesion whose parts are connected and finite or one-ended and such that $E(T)$ has finitely many $\Gamma$-orbits, then $G$ also admits a $\Gamma$-canonical tree-decomposition with the same properties.
\end{claim}

\begin{proof}[Proof of Claim]
 We write $\mathcal V=(V_t)_{t\in V(T)}$.
 Note that $(T,\mathcal V)$ is also a $\Gamma$-canonical tree-decomposition of $G$.
 As $E(T_0)$ has finitely many $\Gamma$-orbits, note that the set $\sg{d(u,v): uv\in E(G^+)\setminus E(G)}$ admits some maximum $k_1\in \mathbb N$. As $E(T)$ has finitely many $\Gamma$-orbits, the set $\sg{d_G(u,v): \exists t\in V(T), uv\in E(G\torso{V_t}) \setminus E(G[V_t])}$ also admits a maximum $k_2\in \mathbb N$. We set $k:=\max(k_1,k_2)$ and let $\mathcal V':=(V'_t)_{t\in V(T)}$ be defined by $V'_t:=B_k(V_t)=\sg{v\in V(G): \exists u\in V_t, d_G(u,v)\leq k}$. We claim that the proof of Lemma \ref{lem: connected-parts} still works here and implies that $(T,\mathcal V')$ is a $\Gamma$-canonical tree-decomposition of $G^+$ of finite adhesion, such that for every $t\in V(T)$, $G^+\torso{V_t}$ is quasi-isometric to $G[V'_t]$. More precisely, every mapping $\pi: V'_t\to V_t$ such that for all $v\in V'_t$, $d_G(\pi(v),v)=d_G(V_t, v)$ defines a quasi-isometry between $G^+[V_t]$ and $G[V'_t]$. In particular, for each $t\in V(T)$, $G[V'_t]$ has at most one end.
\end{proof}

Claim \ref{clm: torsos-plan} allows us to assume without loss of generality that $G^+=G$, i.e., that for each $t\in V(T_0)$ we have $G\torso{V_t}=G[V_t]$. 
For each $t\in V(T_0)$ such that $V_t$ is infinite, $G\torso{V_t}=G[V_t]$ is $3$-connected, thus we can apply Corollary \ref{cor: plan-3conn} to find a $\Gamma_{t}$-canonical tree-decomposition $(T_{t}, \mathcal V_{t})$ of $G[V_{t}]$ with finite adhesion whose parts are connected and have at most one end, and such that $E(T_{t})$ has finitely many $\Gamma_{t}$-orbits. Eventually, \cite[Proposition 7.2]{CTTD} (or \cite[Proposition 3.10, Remark 3.11]{EGL23}) implies that there exists some canonical tree-decomposition $(\widetilde T, \widetilde V)$ refining $(T_0, \mathcal V_0)$, whose parts are connected and either finite or one-ended, and such that $E(T)$ has finitely many $\Gamma$-orbits. The fact that for each $t\in V(\widetilde T)$, $\Gamma_t$ acts quasi-transitively on $G[V_t]$ follows from \cite[Lemma 3.13]{EGL23}.
\end{proof}

 \subsection*{Acknowledgements}

 I would like to thank Matthias Hamann for precious conversations about this work, and in particular for mentioning the reference \cite{EKT20} and pointing out Remark \ref{rem: Z-sums}. I am also grateful to Louis Esperet, Johannes Carmesin and Agelos Georgakopoulos for their careful reading and precious comments on a previous version of this work, which was included in my PhD manuscript. Finally, I would like thank the anonymous reviewers of the journal version for their insightful comments.

\bibliographystyle{alpha}
\bibliography{biblio}

\begin{thebibliography}{HLMR22}

\bibitem[Ant11]{Pichel09}
Yago Antol\'in.
\newblock On {C}ayley graphs of virtually free groups.
\newblock {\em Groups Complexity Cryptology}, 3(2):301--327, 2011.

\bibitem[Bab97]{Babai97}
L{\'a}szl{\'o} Babai.
\newblock The growth rate of vertex-transitive planar graphs.
\newblock In {\em ACM-SIAM Symposium on Discrete Algorithms}, 1997.

\bibitem[CDHS11]{CDHS}
Johannes Carmesin, Reinhard Diestel, Fabian Hundertmark, and Maya Stein.
\newblock Connectivity and tree structure in finite graphs.
\newblock {\em Combinatorica}, 34:11--46, 2011.

\bibitem[CHM22]{CTTD}
Johannes Carmesin, Matthias Hamann, and Babak Miraftab.
\newblock Canonical trees of tree-decompositions.
\newblock {\em Journal of Combinatorial Theory, Series B}, 152:1--26, 2022.

\bibitem[CK23]{CK23}
Johannes Carmesin and Jan Kurkofka.
\newblock Canonical decompositions of 3-connected graphs.
\newblock {\em 2023 IEEE 64th Annual Symposium on Foundations of Computer
  Science (FOCS)}, pages 1887--1920, 2023.

\bibitem[Dro06]{Droms}
Carl Droms.
\newblock Infinite-ended groups with planar {C}ayley graphs.
\newblock {\em Journal of Group Theory}, 9(4):487--496, 2006.

\bibitem[DSS98]{DSS98}
Carl Droms, Brigitte Servatius, and Herman Servatius.
\newblock The structure of locally finite two-connected graphs.
\newblock {\em Electronic Journal of Combinatorics}, 2, 01 1998.

\bibitem[Dun85]{Dunwoody1985}
Martin~J. Dunwoody.
\newblock The accessibility of finitely presented groups.
\newblock {\em Inventiones Mathematicae}, 81:449--458, 1985.

\bibitem[Dun09]{Dunwoody07}
Martin~J. Dunwoody.
\newblock Planar graphs and covers, 2009.

\bibitem[EFW12]{EFW12}
Alex Eskin, David Fisher, and Kevin Whyte.
\newblock Coarse differentiation of quasi-isometries. {I}: {Spaces} not
  quasi-isometric to {Cayley} graphs.
\newblock {\em Annals of Mathematics. Second Series}, 176(1):221--260, 2012.

\bibitem[EGLD24]{EGL23}
Louis Esperet, Ugo Giocanti, and Cl{\'e}ment Legrand-Duchesne.
\newblock The structure of quasi-transitive graphs avoiding a minor with
  applications to the domino problem.
\newblock {\em Journal of Combinatorial Theory. Series B}, 169:561--613, 2024.

\bibitem[EKT22]{EKT20}
Christian Elbracht, Jakob Kneip, and Maximilian Teegen.
\newblock Trees of tangles in infinite separation systems.
\newblock {\em Mathematical Proceedings of the Cambridge Philosophical
  Society}, 173(2):297--327, 2022.

\bibitem[Geo14]{Georgakopoulos14}
Agelos Georgakopoulos.
\newblock Characterising planar {C}ayley graphs and {C}ayley complexes in terms
  of group presentations.
\newblock {\em European Journal of Combinatorics}, 36:282--293, 2014.

\bibitem[Geo17a]{GeoCubic}
Agelos Georgakopoulos.
\newblock {\em The planar cubic {Cayley} graphs}, volume 1190 of {\em Memoirs
  of the American Mathematical Society}.
\newblock Providence, RI: American Mathematical Society (AMS), 2017.

\bibitem[Geo17b]{GeoPlanar2conn}
Agelos Georgakopoulos.
\newblock The planar cubic {Cayley} graphs of connectivity 2.
\newblock {\em European Journal of Combinatorics}, 64:152--169, 2017.

\bibitem[Geo20]{Georgakopoulos_Kleinian}
Agelos Georgakopoulos.
\newblock On planar {Cayley} graphs and {Kleinian} groups.
\newblock {\em Transactions of the American Mathematical Society},
  373(7):4649--4684, 2020.

\bibitem[GH15]{GH15}
Agelos Georgakopoulos and Matthias Hamann.
\newblock The planar {C}ayley graphs are effectively enumerable {I}:
  Consistently planar graphs.
\newblock {\em Combinatorica}, 39:993--1019, 2015.

\bibitem[GH23]{GH22}
Agelos Georgakopoulos and Matthias Hamann.
\newblock The planar {C}ayley graphs are effectively enumerable {II}.
\newblock {\em European Journal of Combinatorics}, 110:103668, 2023.

\bibitem[GH24]{GH24}
Agelos Georgakopoulos and Matthias Hamann.
\newblock A full halin grid theorem.
\newblock {\em Discrete {\&} Computational Geometry}, 2024.

\bibitem[Hal65]{HalinGrid}
Rudolf Halin.
\newblock Über die maximalzahl fremder unendlicher wege in graphen.
\newblock {\em Mathematische Nachrichten}, 30(1-2):63--85, 1965.

\bibitem[Ham15]{HamannCycle}
Matthias Hamann.
\newblock Generating the cycle space of planar graphs.
\newblock {\em The Electronic Journal of Combinatorics}, 22(2):research paper
  p2.34, 8, 2015.

\bibitem[Ham18]{HamannPlanar}
Matthias Hamann.
\newblock Planar transitive graphs.
\newblock {\em The Electronic Journal of Combinatorics}, 25(4):research paper
  p4.8, 18, 2018.

\bibitem[HLMR22]{HamannStallings22}
Matthias Hamann, Florian Lehner, Babak Miraftab, and Tim Rühmann.
\newblock A {S}tallings type theorem for quasi-transitive graphs.
\newblock {\em Journal of Combinatorial Theory, Series B}, 157:40--69, 2022.

\bibitem[Imr75]{Imrich}
Wilfried Imrich.
\newblock On {W}hitney’s theorem on the unique embeddability of 3-connected
  planar graphs.
\newblock In {\em Recent Advances in Graph Theory, Proceedings of the Symposium
  held in Prague}, volume 1974, pages 303--306, 1975.

\bibitem[KPS73]{KPS}
Abraham Karrass, Alfred Pietrowski, and Donald Solitar.
\newblock Finite and infinite cyclic extensions of free groups.
\newblock {\em Journal of the Australian Mathematical Society}, 16(4):458--466,
  1973.

\bibitem[Mac67]{MacBeath}
Angus Macbeath.
\newblock The classification of non-euclidean plane crystallographic groups.
\newblock {\em Canadian Journal of Mathematics}, 19:1192 -- 1205, 1967.

\bibitem[Mac24]{Mac24Planar}
Joseph MacManus.
\newblock A note on quasi-transitive graphs quasi-isometric to planar
  ({Cayley}) graphs.
\newblock Preprint, {arXiv}:2407.13375 [math.{CO}] (2024), 2024.

\bibitem[Mas96]{Maschke}
Heinrich Maschke.
\newblock The representation of finite groups, especially of the rotation
  groups of the regular bodies of three-and four-dimensional space, by cayley's
  color diagrams.
\newblock {\em American Journal of Mathematics}, 18(2):156--194, 1896.

\bibitem[Moh06]{Mohar06}
Bojan Mohar.
\newblock Tree amalgamation of graphs and tessellations of the {Cantor} sphere.
\newblock {\em Journal of Combinatorial Theory. Series B}, 96(5):740--753,
  2006.

\bibitem[MS83]{MS83}
David~E. Muller and Paul~E. Schupp.
\newblock Groups, the theory of ends, and context-free languages.
\newblock {\em Journal of Computer and System Sciences}, 26(3):295--310, 1983.

\bibitem[MS22]{MS22}
Babak Miraftab and Konstantinos Stavropoulos.
\newblock Splitting groups with cubic {Cayley} graphs of connectivity two.
\newblock {\em Algebraic Combinatorics}, 4(6):971--987, 2022.

\bibitem[Ren03]{Renault}
David Renault.
\newblock Enumerating {P}lanar {L}ocally {F}inite {C}ayley {G}raphs.
\newblock {\em Geometriae Dedicata}, 112:25--49, 2003.

\bibitem[Sta68]{Sta68}
John~R. Stallings.
\newblock On torsion-free groups with infinitely many ends.
\newblock {\em Annals of Mathematics. Second Series}, 88:312--334, 1968.

\bibitem[Tho80]{Thomassen80}
Carsten Thomassen.
\newblock Planarity and duality of finite and infinite graphs.
\newblock {\em Journal of Combinatorial Theory, Series B}, 29(2):244--271,
  1980.

\bibitem[Tut84]{Tutte}
William~T. Tutte.
\newblock {\em Graph Theory}.
\newblock Encyclopedia of Mathematics and its Applications. Cambridge
  University Press, 1984.

\bibitem[TW93]{TW}
Carsten Thomassen and Wolfgang Woess.
\newblock Vertex-transitive graphs and accessibility.
\newblock {\em Journal of Combinatorial Theory, Series B}, 58(2):248--268,
  1993.

\bibitem[Whi33]{Whitney}
Hassler Whitney.
\newblock 2-isomorphic graphs.
\newblock {\em American Journal of Mathematics}, 55(1):245--254, 1933.

\bibitem[Wil66]{Wilkie}
H.~C. Wilkie.
\newblock On non-euclidean crystallographic groups.
\newblock {\em Mathematische Zeitschrift}, 91:87--102, 1966.

\bibitem[Woe89]{Woess89}
Wolfgang Woess.
\newblock Graphs and groups with tree-like properties.
\newblock {\em Journal of Combinatorial Theory, Series B}, 47(3):361--371,
  1989.

\bibitem[Woe91]{Woe91}
Wolfgang Woess.
\newblock Topological groups and infinite graphs.
\newblock {\em Discrete Mathematics}, 95(1-3):373--384, 1991.

\bibitem[ZVC80]{Zieschang80}
Heiner Zieschang, Elmar Vogt, and Hans-Dieter Coldewey.
\newblock {\em Surfaces and planar discontinuous groups. {Revised} and expanded
  transl. from the {German} by {J}. {Stillwell}}, volume 835 of {\em Lecture
  Notes in Mathematics}.
\newblock Springer, Cham, 1980.

\end{thebibliography}

\end{document}